\documentclass[11pt,letterpaper,reqno]{amsart}
\usepackage[utf8]{inputenc}
\usepackage[finnish,english]{babel}
\usepackage{amsmath,amsfonts,amssymb, amsthm, setspace, mathtools, hyperref, mathrsfs}
\usepackage{fullpage}
\usepackage{xcolor}
\usepackage{extarrows}
\usepackage{tikz}
\usetikzlibrary{babel}
\usepackage{enumerate}
\usepackage{tikz-cd}
\mathtoolsset{showonlyrefs}
\usepackage{graphicx} 
\usepackage{comment}

\usepackage{caption, subcaption, overpic}
\usepackage{listings}

\usepackage[
giveninits=true,
maxbibnames=99,
backend=bibtex,
style=alphabetic,
doi=false,
url=false
]{biblatex}
\renewbibmacro{in:}{%
  \ifentrytype{article}
    {}
    {\printtext{\bibstring{in}\intitlepunct}}}
\bibliography{references}

\theoremstyle{plain}
\newtheorem{theorem}{Theorem}[section]
\newtheorem{proposition}[theorem]{Proposition}
\newtheorem{lemma}[theorem]{Lemma}
\newtheorem{corollary}[theorem]{Corollary}

\newtheorem{remark}[theorem]{Remark}
\newtheorem{example}[theorem]{Example}

\title{Generic Recovery of Permittivity and Permeability in Anisotropic Maxwell Systems}
\author{ Antonio Cocan, Maarten V. de Hoop, Joonas Ilmavirta, \\ Matti Lassas, and Anthony Várilly-Alvarado}
\address{Department of Mathematics and Statistics, University of Jyv\"askyl\"a, P.O. Box 35, FI-40014 University of Jyv\"askyl\"a, Finland}
\email{antonio.r.popgorea@jyu.fi}
\address{Simons Chair in Computational and Applied Mathematics and Earth Science, Rice University, 6100 S.\ Main St., Houston, TX 77005, USA}
\email{mvd2@rice.edu}
\address{Department of Mathematics and Statistics, University of Jyv\"askyl\"a, P.O. Box 35, FI-40014 University of Jyv\"askyl\"a, Finland}
\email{joonas.ilmavirta@jyu.fi}
\address{Department of Mathematics and Statistics, University of Helsinki, P.O. Box 68 (Gustaf H\"allstr\"omin katu 2B) FI-00014, University of Helsinki, Finland}
\email{Matti.Lassas@helsinki.fi}
\address{Department of Mathematics MS 136, Rice University, 6100 S.\ Main St., Houston, TX 77005, USA}
\email{av15@rice.edu}

\makeatletter
\@namedef{subjclassname@2020}{\textup{2020} Mathematics Subject Classification}
\makeatother

\subjclass[2020]{Primary 78A46; Secondary 14L24, 14E05, 35Q61.}
\date{\today}

\newcommand{\R}{\mathbb{R}}
\newcommand{\PP}{\mathbb{P}}
\newcommand{\C}{\mathbb{C}}

\newcommand{\Z}{\mathbb{Z}}
\newcommand{\V}{\mathbb{V}}

\newcommand{\eps}{\varepsilon}

\newcommand{\hideqed}{\renewcommand{\qed}{}} 

\newcommand{\dd}{\mathrm{d}}

\newcommand{\sslash}{\mathbin{/\mkern-6mu/}}

\DeclareMathOperator{\spec}{Spec}

\DeclareMathOperator{\irr}{Irr}

\DeclareMathOperator{\im}{im}

\DeclareMathOperator{\Sym}{Sym}
\DeclareMathOperator{\ev}{ev}
\DeclareMathOperator{\diag}{diag}

\begin{document}

\begin{abstract} 
We study the inverse problem of recovering the constitutive tensors of a homogeneous anisotropic electromagnetic medium without magnetoelectric coupling (non-chiral) from its Fresnel surface, the characteristic variety of Maxwell's equations governing electromagnetic wave propagation.
For known isotropic permeability, normalized to $\mu = I$, we prove that the Fresnel surface uniquely determines the permittivity tensor $\eps$, and that the associated Fresnel polynomial is reducible precisely when $\eps$ has a repeated eigenvalue.
For general, positive-definite symmetric tensors $(\eps,\mu)$, we prove that the Fresnel polynomial is generically irreducible over $\C$ and we identify the natural gauge symmetry under which it is invariant. 
Using geometric invariant theory, a powerful tool of modern algebraic geometry, we construct an affine quotient of the parameter space by this gauge action and prove that the induced Fresnel-polynomial map is birational onto its image. 
We deduce that, outside a proper real algebraic exceptional set, the real Fresnel surface determines $(\eps,\mu)$ up to gauge. 
This establishes generic uniqueness for the inverse problem and, to our knowledge, provides a new application of affine geometric invariant theory to gauge freedom in a PDE inverse problem.
\end{abstract}

\maketitle

\section{Introduction}
\label{sec:introduction}

An electromagnetic medium is described by its permittivity $\eps$ and its permeability $\mu$, both positive definite and symmetric $3\times3$ real matrices.
These material parameters define a Fresnel polynomial whose vanishing set is the Fresnel surface.
This surface describes the propagation of electromagnetic waves through the medium, so travel-time measurements provide information about it, leading to the inverse problem at the heart of this paper:
Are the material parameters $(\eps,\mu)$ uniquely determined by the Fresnel surface? We answer this question affirmatively, up to a natural gauge freedom that disappears, for example, when a material is magnetically isotropic, i.e., when $\mu=\mu_0I$ is a known scalar matrix. 

We use ideas from algebraic geometry to solve this inverse problem. 
For magnetically isotropic materials, the methods we use are elementary, but in the general case, we leverage more sophisticated tools from modern algebraic geometry, including geometric invariant theory.
To the best of our knowledge, this is the first use of modern geometric invariant theory to account for gauge freedom in PDE inverse problems.
However, the connections between this inverse problem and classical algebraic geometry have deep historical roots, which we review in \S\ref{sec:related-results-algebra}.

It is natural to ask for the further determination of permittivity and permeability \emph{fields} $\eps(x)$ and $\mu(x)$, i.e., for an understanding of how the parameters $(\eps,\mu)$ vary across physical space in an inhomogeneous material. 
Geometrically, this amounts to studying a Fresnel-surface bundle, whose fibres are individual Fresnel surfaces. 
To understand this global object, however, it seems prudent first to study the fibres themselves and their determination from pointwise data. 
Such a study is the objective of this paper.

The ``fibrewise study of electromagnetic geometry'' is analogous to a recent investigation into slowness surfaces arising from anisotropic elasticity~\cite{tonyjoonas2023}.
Structural similarities are unmistakable: material parameters are to be determined from the characteristic variety of a system of hyperbolic partial differential equations arising from a well-known physical model.
The algebraic structures one ends up with are, however, substantially different.

\subsection{Main Results}
\label{sec:main-results}

The material parameters $\eps$ and $\mu$ are collected into a single $6\times6$ block matrix
\begin{equation}
    \label{eq:M-block-matrix}
    M
    =
    \begin{pmatrix}
        \eps&0\\
        0&\mu
    \end{pmatrix}
    .
\end{equation}
In more exotic media, the off-diagonal entries may also differ from zero; i.e., we may have two more $3 \times 3$ matrices $\eta$ and $\xi$ in the off-diagonal blocks, satisfying $\xi = -\eta^T$, which correspond to magnetoelectric coupling, but we shall restrict our study to the block-diagonal case. 
Physically, this corresponds to the absence of chirality in our medium.

For $p\in\R^3$, the linear endomorphism of $\R^3$ given by $v\mapsto p\times v$ yields a matrix
\begin{equation}
    \label{eq:ptimes_matrix}
    {p\times}
    =
    \begin{pmatrix}
        0 & -p_3 & p_2 \\
        p_3 & 0 & -p_1 \\
        -p_2 & p_1 & 0
    \end{pmatrix}
\end{equation}
under the standard basis of $\R^3$. 
Using the $6\times6$ block matrix
\begin{equation}
    \label{eq:Cp_defn}
    C_p
    =
    \begin{pmatrix}
        0 & -p\times \\
        p\times & 0
    \end{pmatrix},
\end{equation}
we define the Fresnel polynomial by
\begin{equation}
    \label{eq:fresnel-polynomial-definition}
    P_{\eps,\mu}(p)
    =
    \det(C_p-M)
\end{equation}
and the Fresnel surface as its zero-locus
\begin{equation}
    \Sigma_{\eps,\mu}
    =
    \V(P_{\eps,\mu})
    :=
    \{
    p\in\R^3;
    P_{\eps,\mu}(p)=0
    \}
    .
\end{equation}

Our first two results together state that in a magnetically isotropic ($\mu=I$) and non-chiral medium, the permittivity~$\eps$ is uniquely determined by the Fresnel surface.

\begin{theorem}
    \label{thm:reducibility-characterization}
    The Fresnel polynomial $P_{\eps,I}$ is reducible if and only if the permittivity $\eps$ has a repeated eigenvalue.
\end{theorem}

\begin{theorem}
    \label{thm:eps-uniqueness}
    Let $\eps_1$ and $\eps_2$ be positive definite, symmetric real $3\times3$ matrices. 
    If $\Sigma_{\eps_1,I}=\Sigma_{\eps_2,I}$, then $\eps_1=\eps_2$.
\end{theorem}

From the conjunction of these two theorems, we deduce that a small Euclidean subset of the Fresnel surface suffices to uniquely determine permittivity.

\begin{corollary}
    \label{thm:eps-uniqueness-subset}
    A nonempty open set, in the usual Euclidean relative topology, of the smooth locus of the Fresnel surface $\Sigma_{\eps,I}\subset\R^3$ uniquely determines the permittivity $\eps$, as long as $\eps$ has no repeated eigenvalues.
\end{corollary}

Consequently, kinematic measurements of a medium providing access to an open subset of the Fresnel surface generically suffice to uniquely determine an unknown $\eps$.

Proceeding to non-magnetically isotropic materials, we give a new proof of an analog of Theorem~\ref{thm:reducibility-characterization}.

\begin{theorem}[Generic Irreducibility]
    \label{thm:generic-irreducibility}
    The Fresnel polynomial $P_{\eps,\mu}$ associated to generic material parameters $(\eps,\mu)$ is irreducible over $\C$.
\end{theorem}

We emphasize that Theorem~\ref{thm:reducibility-characterization} is not a consequence of Theorem~\ref{thm:generic-irreducibility}. 
Indeed, the set of possibly reducible Fresnel polynomials in the proof of Theorem~\ref{thm:generic-irreducibility} could have included all Fresnel polynomials with $\mu = I$ for dimension reasons. 

Theorem~\ref{thm:generic-irreducibility} is implicit in the classical theory of Fresnel wave surfaces due to their tight connection to Kummer surfaces, which are known to be irreducible; see, e.g., Bateman~\cite{bateman1910kummer} and Baekler–Favaro–Itin–Hehl~\cite{Baekler2014}. 
However, not every Fresnel surface is a Kummer surface; this is only true under a genericity assumption. 
For instance, the Fresnel polynomial with $\eps = \mu = I$ is reducible. 

The proof of Theorem~\ref{thm:generic-irreducibility} also uses algebraic geometry, but it is remarkably simple. \\

The matrix $C_p$ depends linearly on $p$ but has only rank 4, because the operator $p\times$ has rank~$2$. 
The Fresnel polynomial $P_{\eps,\mu}\in\R[p_1,p_2,p_3]$ has degree $4$; it is supported on $22$ monomials, all of which have even degree, and its constant term is $\det M = \det(\eps)\det(\mu)$. Fixing a monomial ordering, we identify $P_{\eps,\mu}\in\R[p_1,p_2,p_3]$ with its vector of coefficients, giving a point $P_{\eps,\mu}\in\R^{22}$.
We denote the map from the material parameters $(\eps,\mu)\in\R^{6+6}$ to the Fresnel polynomial by $F\colon\R^{12}\to\R^{22}$.

In Proposition~\ref{prop:fresnel_gauge_freedom} below, we show that
\begin{equation}
    \label{eq:gauge}
    P_{\eps,\mu}
    =
    P_{\mu,\eps}
    =
    P_{\lambda\eps,\lambda^{-1}\mu}
\end{equation}
for all $\eps,\mu\in\R^6$ (symmetric real $3\times3$ matrices) and all $\lambda\in\R\setminus\{0\}$. 
These conditions define the gauge freedom for the Fresnel polynomial, and prompt us to define an equivalence relation $\sim$ on $\R^{12}$ by declaring that $(\eps,\mu)\sim(\eps',\mu')$ when $(\eps,\mu)=(\lambda\eps,\lambda^{-1}\mu)$ or $(\eps,\mu)=(\lambda\mu,\lambda^{-1}\eps)$ for some $\lambda \neq 0$. 
Since the map $F$ is constant on equivalence classes, it factors through the quotient $f\colon\R^{12}/{\sim}\to\R^{22}$.

Empirically, observed kinematic data determines the Fresnel surface, rather than a distinguished normalization of the Fresnel polynomial. 
Thus, a more natural target for the map $f$ is a \emph{projective} space of quartic polynomials, provided the polynomial is nonzero. 
However, there is an $\R$-algebraic subtlety here. 
Over $\C$, two square-free polynomials vanish in the same set exactly when they differ by a non-zero complex scalar. 
Over $\R$, equality of real zero sets does not imply that two square-free polynomials differ by a real scalar, since factors without real zeroes are invisible to the real vanishing locus. 
On the other hand, since a general Fresnel polynomial is irreducible over $\C$ (Theorem~\ref{thm:generic-irreducibility}), and its real zero locus contains smooth real points, any quartic polynomial vanishing on the same nonempty Euclidean open subset of the real Fresnel surface must be a scalar multiple of the Fresnel polynomial (see Lemma~\ref{lem:real-surface-determines-quartic}). 
Therefore, generically, passing from the Fresnel polynomial to the Fresnel surface amounts precisely to projectivizing the coefficient vector.
Write $q\colon\R^{22}\setminus\{0\}\to\R\PP^{21}$ for the quotient map to the projectivization. 

Denote by $\R^{12}_+$ the set of points $(\eps,\mu)$ where $\eps$ and $\mu$ are positive definite. 
The equivalence relation $\sim$ induces an action of $\R_{>0}\rtimes \Z/2\Z$ on $\R^{12}_+$. 
Writing $X:=\R^{12}_+/(\R_{>0}\rtimes \Z/2\Z)$,  $Y:=\R\PP^{21}$ and $\phi=q\circ f\big|_X$, the inverse problem of reconstructing $(\eps,\mu)$ from the Fresnel surface $\Sigma_{\eps,\mu}$ up to the gauge~\eqref{eq:gauge} equivalently asks whether the forward operator $\phi\colon X \to Y$ is injective. 
This is the subject of the central result of our investigation:

\begin{theorem}[Generic Unique Reconstruction]
    \label{thm:generic-uniqueness}
    The material parameters $(\eps,\mu)$ are generically uniquely determined by $\Sigma_{\eps,\mu}$ up to the gauge freedom~\eqref{eq:gauge}.
    More precisely, the map $\phi\colon X \to Y$ is generically injective.
\end{theorem}

Unfortunately, $X$ is only a semi-algebraic set, so generic injectivity in Theorem~\ref{thm:generic-uniqueness} is not a statement about Zariski open subsets. 
But we can be algebro-geometrically precise about what we mean as follows (see~\S\ref{sec:generic_reconstruction_proof}). 
Regard pairs of symmetric $3\times3$ matrices $(\eps,\mu)$ as points of the real affine space
\[
    \Sym_3(\R)\times\Sym_3(\R)\simeq\R^{12}.
\]
Then there exists a proper gauge-invariant real algebraic subset $\mathcal E \subset \R^{12}$ such that, whenever $(\eps,\mu)$ and $(\eps',\mu') \in \R^{12}_+\setminus \mathcal E$, if $\Sigma_{\eps,\mu}=\Sigma_{\eps',\mu'}$, then $(\eps,\mu)$ and $(\eps',\mu')$ differ by a gauge transformation~\eqref{eq:gauge}. 
Therefore, outside a proper real algebraic locus of $\R^{12}$, the Fresnel surface determines the material parameters up to the gauge freedom~\eqref{eq:gauge}.

\subsection{Paper Outline}
\label{sec:outline}

In \S\ref{sec:maxwell}, we review the form of Maxwell's equations and derive the Fresnel polynomial as the characteristic polynomial governing plane-wave propagation. 
We then place our results in the context of the classical algebraic geometry of Fresnel and Kummer surfaces and of inverse problems for Maxwell's equations in \S\ref{sec:related_work}. 

The thrust of the paper begins in \S\ref{sec:gauge_freedom}, where we establish the gauge symmetries of the Fresnel polynomial and formulate the inverse problem modulo these natural equivalences.

We proceed to study the magnetically isotropic case in~\S\ref{sec:magnetically_isotropic_case}, where we characterize the reducibility of the Fresnel polynomial in terms of the eigenvalue multiplicities of the permittivity tensor and prove that the Fresnel surface uniquely determines the permittivity (Theorems~\ref{thm:reducibility-characterization} and~\ref{thm:eps-uniqueness}).

General non-magnetoelectric media are treated in \S\S\ref{sec:nonmagnetoelectric_irreducibility}--\ref{sec:nonmagnetoelectric_generic_uniqueness}. 
In \S\ref{sec:nonmagnetoelectric_irreducibility}, we prove the generic irreducibility of the Fresnel polynomial (Theorem~\ref{thm:generic-irreducibility}), and we establish the generic uniqueness theorem (Theorem~\ref{thm:generic-uniqueness}) in \S\ref{sec:nonmagnetoelectric_generic_uniqueness}. 
We first construct the relevant algebraic quotient under the gauge action in~\S\ref{sec:algebraic_quotient}, then analyze the resulting rational map by passing to a suitable slice in~\S\ref{sec:slicing}, compactifying the parameter space, and resolving its indeterminacies in~\S\ref{sec:compactification_and_blow-up}. 
Theorem~\ref{thm:Fresnel_polynomial_map_birational} is the key algebro-geometric result that allows us to prove Theorem~\ref{thm:generic-uniqueness}. 
We conclude by relating projective classes of Fresnel polynomials to real Fresnel surfaces, thereby proving generic reconstruction of material parameters from the Fresnel surface up to the natural gauge freedom.

\subsection*{Use of AI}
After completing a draft of this paper, we asked OpenAI's GPT5.6-Sol Pro to audit the manuscript for the correctness of claims and arguments and logical holes. 
Its response helped to bolster some arguments that were essentially correct, but could have used more detail in \S\S\ref{sec:gauge_freedom}--\ref{sec:magnetically_isotropic_case}; it suggested Theorem~\ref{thm:irr_locus}, greatly simplifying our original proof of Theorem~\ref{thm:generic-irreducibility}, which followed along the methods in~\cite{tonyjoonas2023}; and it helped significantly improve the exposition in \S\ref{sec:nonmagnetoelectric_generic_uniqueness}. 
The AI also wrote \textsc{magma} code to perform further consistency checks on Examples~\ref{ex:irreducible_fresnel_polynomial} and~\ref{ex:special-fibre}, although our original code already established all the claims we needed. 
The paper remains human-written.

\subsection*{Acknowledgements}

A.\ Cocan and J.\ Ilmavirta were supported by the Research Council of Finland (Flagship of Advanced Mathematics for Sensing Imaging and Modelling grant 359208; Centre of Excellence of Inverse Modelling and Imaging grant 353092; and other grants 351665, 351656, 358047, 360434) and a Väisälä project grant by the Finnish Academy of Science and Letters. 
M.\ V.\ de Hoop was supported by the Simons Foundation under the MATH + X program, the National Science Foundation under grant DMS-1815143, and the corporate members of the Geo-Mathematical Imaging Group at Rice University. 
M.\ Lassas was partially supported by the ERC Advanced Grant project 101097198 (Inverse PDE) of the European Research Council and the FAME flagship of the Research Council of Finland (grant 359186).  
A.\ V\'arilly-Alvarado was supported by NSF grant DMS-2302231. 
The views and opinions expressed are those of the authors only and do not necessarily reflect those of the funding agencies or the EU.

Computer calculations for this paper were done using the \textsc{magma} computer algebra system~\cite{magma}.

\section{Review: Maxwell's equations}
\label{sec:maxwell}

We briefly recall the form of Maxwell's equations we use and explain the origin of the Fresnel polynomial. 
In our presentation, we assume the material parameters are homogeneous (independent of the point), non-dispersive (independent of frequency), and free of free charge (so that the current density vanishes).

In a general local linear medium, the electric field
$E(x,t)\colon \R^3\times \R \to \R^3$ and magnetic field
$H(x,t)\colon \R^3\times \R \to \R^3$ are related to the electric flux density
$D(x,t)\colon \R^3\times \R \to \R^3$ and the magnetic flux density
$B(x,t)\colon \R^3\times \R \to \R^3$ by the constitutive relations
\begin{equation}
    \label{eq:constitutive-general}
    \begin{pmatrix}
        D\\
        B
    \end{pmatrix}
    =
    M
    \begin{pmatrix}
        E\\
        H
    \end{pmatrix},
    \qquad\text{where }
    M=
    \begin{pmatrix}
        \eps & \xi\\
        \eta & \mu
    \end{pmatrix},
\end{equation}
and the $3\times3$ matrices $\eps,\mu,\eta,\xi$ are constant. 
Throughout most of the paper, we restrict attention to nonchiral media, for which $\eta=\xi=0$, so that $M$ takes the diagonal block form~\eqref{eq:M-block-matrix}, with $\eps$ and $\mu$ symmetric and positive definite. 
The fields $E$, $H$, $D$ and $B$ satisfy Maxwell's equations
\begin{equation}
    \label{eq:maxwell-time}
    \begin{aligned}
        \nabla\times E &= -\partial_t B,\\
        \nabla\times H &= \phantom{-}\partial_t D,\\
        \nabla\cdot D &= 0,\\
        \nabla\cdot B &= 0.
    \end{aligned}
\end{equation}

In a homogeneous medium, the propagation of electromagnetic waves can be understood using the plane-wave ansatz
\[
    (E,H)(x,t)
    =
    (E_0,H_0)e^{i(p\cdot x-\omega t)}
\]
with the wave (co)vector $p\in\R^3$ and frequency $\omega$.
Upon taking Fourier transforms, we replace
\[
    \nabla\mapsto ip,
    \qquad
    \partial_t\mapsto -i\omega
\]
in~\eqref{eq:maxwell-time} and obtain the system
\begin{equation}
    \label{eq:plane-wave-system}
    \begin{aligned}
        p\times E &= \omega B,\\
        p\times H &= -\omega D.
    \end{aligned}
\end{equation}
We use the matrix $C_p$ in~\eqref{eq:Cp_defn} to rewrite the system more compactly as
\begin{equation}
    \label{eq:maxwell-symbol}
    \left(
    C_p-\omega M
    \right)
    \begin{pmatrix}
        E_0\\
        H_0
    \end{pmatrix}
    =0.
\end{equation}

A nontrivial plane wave exists precisely when the matrix in
\eqref{eq:maxwell-symbol} is singular, and the polarization vector $(E_0,H_0)$ is in its kernel. 
We thus get the condition
\[
    \det(C_p-\omega M)=0.
\]
There is a universal factor of $\omega^2$:
\[
    \det(C_p-\omega M) = \omega^2\widetilde P_{\eps,\mu}(\omega,p),
\]
where $\widetilde P_{\eps,\mu}$ is homogeneous of degree $4$. 
The \emph{reduced propagating characteristic surface} is $\V(\widetilde P_{\eps,\mu})\subset \PP^3$, and the affine Fresnel surface is the chart where we normalize $\omega = 1$:
\[
    \Sigma_{\eps,\mu}
    :=
    \{p\in\R^3:\widetilde P_{\eps,\mu}(1,p)=0\}
\]
This leads to the Fresnel polynomial 
\[
    P_{\eps,\mu}(p) := \widetilde P_{\eps,\mu}(1,p) = \det(C_p - M).
\]
cutting out the affine Fresnel surface.

The assumptions of homogeneity and lack of charge were convenient for global analysis but, in fact, unnecessary for understanding the propagation of electromagnetic waves.
If the material parameters vary, one can use essentially the same procedure as above and show that $C_p-\omega M(x)$ is the principal symbol of the Maxwell equations.
The densities of electric charge and current are lower order effects and do not contribute to the principal symbol.
The set of spacetime covectors $(p,\omega)$ for which this symbol has a kernel is exactly the Fresnel surface.
In other words, the Fresnel surface is the characteristic variety of Maxwell's equations. 
Singularities of electromagnetic waves propagate along the bicharacteristic flow generated by the Hamiltonian of $P_{\eps,\mu}$, and the Fresnel surface describes the admissible wave covectors at each point of the medium.
This is a form of wave--particle duality in the high-frequency limit that follows from microlocal analysis; see, e.g., ~ \ cite{Dencker} for details on the propagation of singularities for well-behaved hyperbolic systems.

It is easier to think of our results in the context of homogeneous media, since the Fresnel surface is a single surface rather than varying from point to point, and it can be obtained directly from travel-time measurements.
This reduction is exactly the same as in~\cite{tonyjoonas2023} for the elastic wave equation, while the algebraic nature of the characteristic variety is quite different.
The results we present here lay the foundations for understanding gauge freedoms in the permittivity and permeability fields $\eps(x)$ and $\mu(x)$, too.

The factor $\omega^2$ can be seen to correspond to two polarizations having zero frequency ($\omega=0$ is a double root of the determinant).
These polarizations point along the wave vector but do not propagate, and it is physically sound to remove these spurious roots by factoring out the $\omega^2$.
This factor can also be seen to stem from the rank-deficiency of $C_p$.

\section{Review: Related Work}
\label{sec:related_work}

\subsection{Algebraic Connections}
\label{sec:related-results-algebra}

The study of Fresnel surfaces originated with Fresnel's 1822 work on light propagation in anisotropic crystals~\cite{Dolgachev}. 
In birefringent media, a quartic dispersion relation constrains admissible wave covectors; its zero set defines the Fresnel surface. 
Classical investigations by Kummer in 1864 revealed a deep connection between these optical wave surfaces and quartic surfaces with 16 nodes in algebraic geometry, now known as Kummer surfaces. 
The singular points of these surfaces govern physically significant propagation phenomena, like conical refraction. 
Bateman \cite{bateman1910kummer}, Hudson \cite{hudson1905kummer}, and Dolgachev~\cite{Dolgachev} provide historical and modern treatments of this connection.

The geometric theory of electromagnetic wave propagation has undergone substantial development in recent decades through the framework of local and linear constitutive laws. 
Notably, Hehl and Obukhov formulated premetric electrodynamics, where the Fresnel surface emerges naturally as the characteristic variety of Maxwell's equations \cite{hehlobukhov2003}. 
Subsequent studies \cite{bergaminfavaro2014,lindellfavarobergamin2011,Baekler2014,FavaroHehl2016} clarified this Fresnel–Kummer link, proving that several physically relevant classes of skewon-free linear constitutive tensors generate Fresnel surfaces projectively equivalent to classical Kummer surfaces, and detailing the resulting singularity structures. 
These results underscore how the Fresnel surface's geometry encodes critical information about the underlying constitutive parameters.

The singularities of Fresnel surfaces have been studied extensively both from physical and algebraic-geometric viewpoints. 
In generic local and linear electromagnetic media, isolated singular points on the Fresnel surface reflect algebraic constraints on the constitutive tensor. 
Favaro and Hehl classified and visualized these singularities \cite{favaro2011,favaro2016}.

Our work also draws on recent developments in anisotropic elasticity for inspiration and techniques. 
In that domain, the Christoffel equation determines a slowness surface analogous to the Fresnel surface in electromagnetism. 
Algebraic-geometric techniques were used to tackle questions of irreducibility, singularities, and the recovery of elastic parameters from these wave surfaces~\cite{tonyjoonas2023}. 
Yet, whereas elasticity phenomena typically give rise to sextic surfaces, Maxwell's equations naturally yield quartic Fresnel surfaces with a substantially different geometric structure.

From the perspective of inverse problems, the problems studied in this article differ fundamentally from the classical inverse boundary value problems for Maxwell's equations. 
Rather than extracting material parameters from boundary measurements, we investigate whether the constitutive parameters can be recovered from the characteristic geometry of the Maxwell system itself. 
In this sense, our work is closer to the study of inverse problems for characteristic varieties and dispersion relations. 
The main results show that, under suitable assumptions, the Fresnel surface uniquely determines the constitutive parameters, thereby establishing a direct correspondence between the geometry of wave propagation and the underlying electromagnetic medium.

\newpage

\subsection{Inverse Problems Connection}
\label{sec:related-results-inverse}

Inverse problems for Maxwell's equations concern the recovery of electromagnetic material parameters from boundary measurements. In the time-harmonic setting, one considers
\begin{equation}
    \label{eq:maxwell-euclidean}
    \nabla\times E=i\omega \mu H,
    \qquad
    \nabla\times H=-i\omega \gamma E,
    \qquad
    \gamma=\varepsilon+i\sigma/\omega,
\end{equation}
where $\omega>0$ is fixed, $\varepsilon$ is the electric permittivity, $\mu$ is the magnetic permeability, and $\sigma$ is the conductivity. 
In the geometric formulation on an oriented Riemannian three-manifold $(M,g)$, the equations become
\begin{equation}\label{eq:maxwell-forms}
    {\star}\dd E=i\omega\mu H,
    \qquad
    {\star}\dd H=-i\omega\varepsilon E,
\end{equation}
and the boundary measurements are encoded by the admittance map
\begin{equation}\label{eq:admittance}
    \Lambda:f=tE\longmapsto tH,
\end{equation}
where $t$ denotes the tangential trace. 
This invariant formulation is particularly well suited to anisotropic constitutive laws.

The fixed-frequency Maxwell inverse problem is closely related to Calder\'on's inverse conductivity problem. 
In the formal zero-frequency limit one recovers the conductivity equation, and many uniqueness proofs rely on the complex geometrical optics method introduced by Sylvester and Uhlmann for Calder\'on's problem \cite{calderon1980inverse,sylvesteruhlmann1987global}. 
Early work on the Maxwell inverse boundary value problem includes the papers of Somersalo, Isaacson and Cheney \cite{somersalo1992linearized} and Sun and Uhlmann \cite{sunuhlmann1992maxwell}. 
The first global uniqueness theorem for smooth isotropic electromagnetic parameters was established by Ola, P\"aiv\"arinta and Somersalo \cite{ola1993electrodynamics}.

A central analytical idea is to augment the first-order Maxwell system with auxiliary scalar fields and reduce it to a Dirac-type system
\begin{equation}
    \label{eq:dirac}
    (P-k+W)Y=0,
\end{equation}
whose square yields a matrix Schr\"odinger equation
\begin{equation}
    \label{eq:schroedinger}
    (P-k+W)(P+k-W^t)Z=(-\Delta-k^2+Q)Z=0.
\end{equation}
This reduction, developed by Ola and Somersalo using generalized Sommerfeld potentials \cite{ola1996sommerfeld}, has become a standard tool for analyzing Maxwell inverse problems. 
Global uniqueness has subsequently been extended to lower regularity coefficients by Caro and Zhou \cite{carozhou2014global}.

The anisotropic problem is substantially more delicate because of the natural geometric gauge invariance. 
Kenig, Salo and Uhlmann proved uniqueness for time-harmonic Maxwell equations on admissible Riemannian three-manifolds and, in Euclidean domains, for matrix-valued coefficients belonging to a common conformal class associated with an admissible metric \cite{kenig2011anisotropicmaxwell}. 
Their proof combines the Dirac--Schr\"odinger reduction with complex geometrical optics constructions on admissible manifolds, extending the approach of Dos Santos Ferreira, Kenig, Salo and Uhlmann \cite{dossantosferreira2009limiting}. 
In contrast, singular or degenerate anisotropic coefficients exhibit the nonuniqueness phenomena of transformation optics and invisibility cloaking \cite{greenleaf2007fullwave,greenleaf2009invisibility}.

\section{Gauge Freedom}
\label{sec:gauge_freedom}

Recall that, quite generally, the material parameters of an electromagnetic medium are collected into a $6\times6$ block matrix
\begin{equation}
    M
    =
    \begin{pmatrix}
        \eps&\xi\\
        \eta&\mu
    \end{pmatrix}
\end{equation}
composed of $3\times 3$ real blocks, such that $\eps$ and $\mu$ are symmetric, and $\xi = -\eta^T$. 
In this generality, the Fresnel polynomial is defined by
\begin{equation}
    P_{\eps,\mu}^{\eta,\xi}(p)
    =
    \det(C_p-M),
\end{equation}
where $C_p$ is the matrix~\eqref{eq:Cp_defn}. 
Let $A = \mathbb{R}[\eps_{1}, \ldots, \eps_{6}, \mu_{1}, \ldots, \mu_{6}]$. 
A calculation shows that, as a polynomial in $A[p_1,p_2,p_3]$, the Fresnel polynomial is supported on the 22 monomials
\begin{equation}
    \begin{aligned}
        &\text{degree $4$: } p_1^4, p_1^3 p_2, p_1^3 p_3, p_1^2 p_2^2, p_1^2 p_2 p_3, p_1^2 p_3^2, p_1 p_2^3, p_1 p_2^2 p_3, p_1 p_2 p_3^2, p_1 p_3^3, p_2^4, p_2^3 p_3, p_2^2 p_3^2, p_2 p_3^3, p_3^4, \\
        &\text{degree $2$: }p_1^2, p_1 p_2, p_1 p_3, p_2^2, p_2 p_3, p_3^2, \\
        &\text{degree $0$: } 1.
    \end{aligned}
\end{equation}

To what extent does the Fresnel surface determine a material's parameters?
That is, does the Fresnel polynomial $p\mapsto P_{\epsilon,\mu}^{\eta,\xi}(p)$ determine $\eps$, $\mu$, $\eta$, and $\xi$? 
Not in this generality: There is a gauge freedom to this question, as follows.

\begin{proposition}
    \label{prop:fresnel_gauge_freedom}
    For $t\in \R\setminus\{0\}$ we have
    \begin{equation}
        P_{\epsilon,\mu}^{\eta,\xi}(p)
        =
        P_{t\epsilon,t^{-1}\mu}^{\eta,\xi}(p),
        \qquad\text{and}\qquad
        P_{\epsilon,\mu}^{\eta,\xi}(p)
        =
        P_{\mu,\epsilon}^{-\xi,-\eta}(p)
        .
    \end{equation}
\end{proposition}

\begin{proof}
    Let
    \begin{equation}
        L_t
        =
        \begin{pmatrix}
            t I & 0 \\
            0 & I
        \end{pmatrix},
        \qquad
        R_t
        =
        \begin{pmatrix}
            I & 0 \\
            0 & t^{-1} I
        \end{pmatrix},
        \qquad
        \text{and}
        \qquad
        J
        =
        \begin{pmatrix}
            0 & I \\
            -I & 0
        \end{pmatrix}
        ,
    \end{equation}
    all understood as $6\times6$ matrices written using $3\times 3$ blocks. We have $\det(L_t)\det(R_t)=\det(J)=1$, $L_t C_p R_t = C_p$, and
    \[
        L_t
        M
        R_t = 
        \begin{pmatrix}
            t\eps & \xi \\
            \eta & t^{-1}\mu
        \end{pmatrix}
    \]
    Taking determinants of this last equality gives $P_{\epsilon,\mu}^{\eta,\xi}(p) = P_{t\epsilon,t^{-1}\mu}^{\eta,\xi}(p)$.
    We also have
    \begin{equation}
        J^{-1}(C_p - M)J = C_p - 
        \begin{pmatrix}
            \mu & -\eta \\
            -\xi & \eps
        \end{pmatrix}.
    \end{equation}
    Taking determinants we get $P_{\epsilon,\mu}^{\eta,\xi}(p) = P_{\mu,\epsilon}^{-\xi,-\eta}(p)$.
\end{proof}

\section{The Magnetically Isotropic Case}
\label{sec:magnetically_isotropic_case}

Throughout this section $\mu=I$, so we omit it from the notation. 
To condense further, we use the shorthand Voigt notation
\begin{equation}
    11 \rightsquigarrow 1, 22 \rightsquigarrow 2, 33 \rightsquigarrow 3, 23, 32 \rightsquigarrow 4, 13, 31 \rightsquigarrow 5, 12, 21 \rightsquigarrow 6
\end{equation}
to write permittivity as
\begin{equation}
    \eps
    =
    \begin{pmatrix}
        \eps_{1} & \eps_{6} & \eps_{5} \\
        \eps_{6} & \eps_{2} & \eps_{4} \\
        \eps_{5} & \eps_{4} & \eps_{3}
    \end{pmatrix}
    .
\end{equation}
With these conventions, the Fresnel polynomial~\eqref{eq:fresnel-polynomial-definition} is given by
\begin{equation}
    \label{eq:fresnel_polynomial_eta_xi_zero}
    \begin{aligned}
        P_{\eps, I} &= \eps_1 p_1^4 + 2\eps_6 p_1^3 p_2 + 2\eps_5 p_1^3 p_3 + (\eps_1 + \eps_2) p_1^2 p_2^2 + 2\eps_4 p_1^2 p_2 p_3 + (\eps_1 + \eps_3) p_1^2 p_3^2 \\
                    &+ (-\eps_1 \eps_2 - \eps_1 \eps_3 + \eps_5^2 + \eps_6^2) p_1^2 + 2\eps_6 p_1 p_2^3 + 2\eps_5 p_1 p_2^2 p_3 + 2\eps_6 p_1 p_2 p_3^2 \\
                    &+ (-2\eps_3 \eps_6 + 2\eps_4 \eps_5) p_1 p_2 + 2\eps_5 p_1 p_3^3 + (-2\eps_2 \eps_5 + 2\eps_4 \eps_6) p_1 p_3 + \eps_2 p_2^4 \\
                    &+ 2\eps_4 p_2^3 p_3 + (\eps_2 + \eps_3) p_2^2 p_3^2 + (-\eps_1 \eps_2 - \eps_2 \eps_3 + \eps_4^2 + \eps_6^2) p_2^2 + 2\eps_4 p_2 p_3^3 \\
                    &+ (-2\eps_1 \eps_4 + 2\eps_5 \eps_6) p_2 p_3 + \eps_3 p_3^4 + (-\eps_1 \eps_3 - \eps_2 \eps_3 + \eps_4^2 + \eps_5^2) p_3^2 \\
                    &+ \eps_1 \eps_2 \eps_3 - \eps_1 \eps_4^2 - \eps_2 \eps_5^2 - \eps_3 \eps_6^2 + 2\eps_4 \eps_5 \eps_6
                    .
    \end{aligned}
\end{equation}

\subsection{Irreducibility Results}

\begin{lemma}
    \label{lma:compact}
    If $\eps$ and $\mu$ are positive definite, then the Fresnel surface $\Sigma_{\eps,\mu}$ is compact in the Euclidean topology.
\end{lemma}

\begin{proof}
    It is closed, so it suffices to prove boundedness. 
    For $u\in S^2$, put
    \[
        B_u
        =
        M^{-1/2}C_uM^{-1/2}.
    \]
    The matrix $B_u$ is symmetric and congruent to $C_u$. It has two positive eigenvalues, two negative eigenvalues, and two zero eigenvalues. 
    The ordered eigenvalues of a symmetric matrix depend continuously on its entries. 
    Since the map $u\longmapsto B_u$ is continuous and $S^2$ is compact, there is a constant $c>0$ such that every nonzero eigenvalue of every $B_u$ has absolute value at least $c$. 
    Write $p=tu$ with $u\in S^2$. 
    If $p\in\Sigma_{\eps, \mu}$, then $1$ is an eigenvalue of
    \[
        M^{-1/2}C_{tu}M^{-1/2}
        =
        tB_u.
    \]
    Thus $|t|\le c^{-1}$. Hence $\Sigma_{\eps, \mu}$ is bounded.
\end{proof}

\begin{lemma}
    \label{lem:fresnel_factorization_quadratics}
    Let $M$ be a real symmetric positive definite matrix. If $P_M(p)=\det(C_p-M)$ is reducible over $\C$, then it is a product of two real quadratic polynomials.
\end{lemma}

\begin{proof}
    For every $u\in S^2$, the matrix $B_u$ in the proof of Lemma~\ref{lma:compact} has four nonzero real eigenvalues. 
    Hence $P_M(tu)$ has four real roots, counted with multiplicity. 
    Let $f\in\C[p_1,p_2,p_3]$ be an irreducible factor of $P_M$. 
    Since $P_M(0)=\det M\neq0,$ we may normalize $f(0)=1$. 
    For generic real $u$, the degree of $f(tu)$ equals $\deg f$, and all of its roots are among the real roots of $P_M(tu)$. 
    Its constant term is~$1$, and none of those roots are~$0$, so its leading coefficient, and hence all its coefficients, are real. 
    Writing $f=\sum_j f_j$ as a sum of homogeneous parts, this says that every $f_j(u)$ is real for $u$ in a Euclidean-open set. 
    The imaginary part of $f_j$ is a real polynomial vanishing on that open set, so it is identically zero. 
    Hence $f$ has real coefficients.

    A nonconstant real linear factor would vanish on an unbounded affine plane contained in $\Sigma_M$, contradicting Lemma~\ref{lma:compact}. 
    Since $\deg P_M=4$, every proper factorization therefore has degrees $2+2$.
\end{proof}

For later use, when $\mu=I$ one has the useful identity
\begin{equation}
    \label{eq:fresnel-mu-identity}
    P_{\eps,I}(p)
    =
    |p|^2\,p^T\eps p
    -
    p^T\!\bigl((\operatorname{tr}\eps)\eps-\eps^2\bigr)p
    +
    \det\eps.
\end{equation}
It agrees term-by-term with the displayed Voigt-coordinate formula. 
In an orthonormal eigenbasis, writing $\eps=\diag(\eps_1,\eps_2,\eps_3)$ and $p=(p_{1},p_{2},p_{3}),$ this becomes
\begin{equation}
    \label{eq:fresnel-diagonal}
    P
    =
    sq
    -
    \eps_1(\eps_2+\eps_3)p_{1}^2
    -
    \eps_2(\eps_1+\eps_3)p_{2}^2
    -
    \eps_3(\eps_1+\eps_2)p_{3}^2
    +
    \eps_1\eps_2\eps_3,
\end{equation}
where
\[
    s=p_{1}^2+p_{2}^2+p_{3}^2,
    \qquad
    q=\eps_1p_{1}^2+\eps_2p_{2}^2+\eps_3p_{3}^2.
\]

\begin{proof}[Proof of Theorem~\ref{thm:reducibility-characterization}]
    The assertion is invariant under an orthogonal change of coordinates, so we may take $\eps=\diag(\eps_1,\eps_2,\eps_3),$ with $\eps_1,\eps_2,\eps_3>0$. 
    Assume first that $P=P_{\eps,I}$ is reducible. 
    By Lemma~\ref{lem:fresnel_factorization_quadratics}, write $P=gh,$ with $g,h\in\R[p_{1},p_{2},p_{3}]$ of degree two. 
    Let $g_i,h_i$ denote their homogeneous parts of degree $i$. 
    The degree-four part of $P$ is $sq$. If $\eps_1=\eps_2=\eps_3$, there is already a repeated eigenvalue, so suppose $q$ is not a scalar multiple of $s$. 
    The positive definite quadrics $s$ and $q$ are irreducible and nonassociate in the UFD $\R[p_{1},p_{2},p_{3}]$. 
    Consequently, after interchanging $g$ and $h$ and rescaling them by inverse constants, we may assume that
    \[
        g_2=s,
        \qquad\text{and}\qquad
        h_2=q.
    \]

    The degree-three part of $gh$ is zero, so
    \[
        sh_1+qg_1=0.
    \]
    Since $\gcd(s,q)=1$, the polynomial $s$ divides $g_1$. But $\deg g_1\le1$, so $g_1=0$; similarly $h_1=0$. Thus
    \[
        g=s+\gamma,
        \qquad
        h=q+\delta,
    \]
    for constants $\gamma,\delta\in\R$. 
    Comparing the degree-two part with~\eqref{eq:fresnel-diagonal} gives
    \begin{equation}
        \begin{aligned}
            \delta+\gamma \eps_1 &= -\eps_1(\eps_2+\eps_3), \\
            \delta+\gamma \eps_2 &= -\eps_2(\eps_1+\eps_3), \\
            \delta+\gamma \eps_3 &= -\eps_3(\eps_1+\eps_2).
        \end{aligned}
    \end{equation}
    Subtracting pairs of equations yields
    \begin{equation}
        \begin{aligned}
            (\eps_1-\eps_2)(\gamma+\eps_3) &= 0, \\
            (\eps_1-\eps_3)(\gamma+\eps_2) &= 0, \\
            (\eps_2-\eps_3)(\gamma+\eps_1) &= 0.
        \end{aligned}
    \end{equation}
    If $\eps_1,\eps_2,\eps_3$ were pairwise distinct, these equations would force $\gamma=-\eps_1=-\eps_2=-\eps_3,$ a contradiction. 
    Hence $\eps$ has a repeated eigenvalue.
    
    Conversely, suppose $\eps_2=\eps_3$. 
    A direct multiplication gives
    \begin{equation}
        \label{eq:repeated-eigenvalue-factorization}
        P_{\diag(\eps_1,\eps_2,\eps_2),I}
        =
        (s-\eps_2)\bigl(\eps_1p_{1}^2+\eps_2(p_{2}^2+p_{3}^2)-\eps_1\eps_2\bigr).
    \end{equation}
    If $\eps_1=\eps_2=\eps_3$, this specializes to
    \[
        P_{\eps_1I,I}
        =
        \eps_1(s-\eps_1)^2.
    \]
    Thus, a repeated eigenvalue implies reducibility.
\end{proof}

\subsection{Eigenvalues of $\eps$ and their Multiplicities}

Given a hypersurface $S\subset\R^3$, denote by $R(S)\subset S$ the set of points $p\in S$ where $S$ is not smooth or where $p$ and the normal vector of $S$ at $p$ are linearly dependent.
The set of these ``radial points'' helps identify $\eps$. 
In the two lemmas that follow $\eps$ is a positive definite $3\times3$ symmetric matrix. 
The permittivity $\eps$ can have three distinct eigenvalues, two distinct eigenvalues, or one eigenvalue. 
We consider each case separately.

\begin{lemma}
    \label{lma:eps-3-eigenvalues}
    If $\eps$ has three distinct eigenvalues, then $R(\Sigma_{\eps})$ consists of three circles, each centered at the origin and orthogonal to an eigenvector, whose squared radius equals the corresponding eigenvalue.
\end{lemma}

\begin{proof}
    The claim is rotation invariant, so we may act by a rotation to diagonalize $\eps$ so that
    \begin{equation}
        \label{eq:eps-diag-matrix}
        \eps
        =
        \begin{pmatrix}
            \eps_1 & 0 & 0 \\
            0 & \eps_2 & 0 \\
            0 & 0 & \eps_3
        \end{pmatrix}
        .
    \end{equation}
    By assumption, the three diagonal entries are distinct.
    The Fresnel polynomial is now
    \begin{equation}
    \label{eq:eps-diag-polynomial}
        \begin{split}
            P_{\eps,I}(p_1,p_2,p_3)
            &=
            \eps_3p_3^4
            +(\eps_3+\eps_2)p_2^2p_3^2
            +(\eps_3+\eps_1)p_1^2p_3^2
            -(\eps_2\eps_3+\eps_1\eps_3)p_3^2
            \\&\quad
            +\eps_2p_2^4
            +(\eps_2+\eps_1)p_1^2p_2^2
            -(\eps_2\eps_3+\eps_1\eps_2)p_2^2
            +\eps_1p_1^4
            \\&\quad
            -(\eps_1\eps_3+\eps_1\eps_2)p_1^2
            +\eps_1\eps_2\eps_3
            .
        \end{split}
    \end{equation}
    By positive definiteness $P_\eps(0,0,0)=\eps_1\eps_2\eps_3>0$ so the origin does not lie on the Fresnel surface.
    
    A point $(p_1,p_2,p_3)\in R(\Sigma_{\eps})$ is characterized among points of $\Sigma_{\eps}$ by the property that the vectors $(p_1,p_2,p_3)$ and $\nabla P_\eps(p_1,p_2,p_3)$ are linearly dependent.
    This is equivalent to the vanishing of all $2\times2$ minors of the $2\times 3$ matrix with rows $(p_1,p_2,p_3)$ and $\nabla P_\eps(p_1,p_2,p_3)$.
    These minors are
    \begin{align}
        \label{eq:eps-diag-minor1}
        m_1
        &:=
        (\eps_2-\eps_3)p_2p_3(p_1^2 + p_2^2 + p_3^2 - \eps_1), \\
        \label{eq:eps-diag-minor2}
        m_2
        &:=
        (\eps_1-\eps_3)p_1p_3(p_1^2 + p_2^2 + p_3^2 - \eps_2), \\
        \label{eq:eps-diag-minor3}
        \text{and } m_3
        &:=
        (\eps_1-\eps_2)p_1p_2(p_1^2 + p_2^2 + p_3^2 - \eps_3)
        .
    \end{align}
    There are two cases where all these minors vanish:
    \begin{enumerate}
        \item
        One of the three coordinates vanishes and the rest are on a circle defined by its eigenvalue.
        For example, $p_3=0$ and $p_1^2 + p_2^2 = \eps_3$.
        \smallskip
    
        \item
        Two of the three coordinates $(p_1,p_2,p_3)$ vanish.
    \end{enumerate}
    
    In the first case, substituting $\eps_3=p_1^2 + p_2^2$ and $p_3=0$ into the Fresnel polynomial makes it vanish.
    Therefore, all the points on this circle are contained in $\Sigma_{\eps}$.
    This produces the three claimed circles.
    
    In the second case, assume without loss of generality that $p_2=p_3=0$.
    Then $P_\eps(p_1,0,0)=\eps_1(p_1^2-\eps_2)(p_1^2-\eps_3)$ and the roots are clearly $p_1=\pm\sqrt{\eps_2}$ and $p_1=\pm\sqrt{\eps_3}$.
    The $12$ points produced like this are exactly the intersection points of the three circles with the coordinate axes, so all of $R(\Sigma_{\eps})$ is covered by the first case.
\end{proof}

\begin{lemma}
    \label{lma:eps-2-eigenvalues}
    If $\eps$ has two distinct eigenvalues, then $R(\Sigma_{\eps})$ consists of a sphere whose radius is the square root of the repeated eigenvalue, and a circle in the $2$-dimensional eigenspace whose radius is the square root of the simple eigenvalue.
\end{lemma}

\begin{proof}
    The setting is very similar to the one in the proof of lemma~\ref{lma:eps-3-eigenvalues}:
    The permittivity can be taken to be as in~\eqref{eq:eps-diag-matrix}, the Fresnel polynomial is still given by~\eqref{eq:eps-diag-polynomial} and the characterization of radial points in terms of the minors in~\eqref{eq:eps-diag-minor1}, \eqref{eq:eps-diag-minor2}, and~\eqref{eq:eps-diag-minor3} is still valid.
    The difference is that, without loss of generality, $\eps_1\neq\eps_2=\eps_3$.

    This change implies that $m_1=0$ at all points.
    Solving $m_2=m_3=0$ splits into three cases:
    \begin{enumerate}
        \item
        $p_1=0$.
        \smallskip
        
        \item
        $p_3=p_2=0$.
        \item
        \smallskip
        
        $p_1^2+p_2^2+p_3^2=\eps_2$.
    \end{enumerate}
    
    In the first case the Fresnel polynomial becomes $\eps_2(p_3^2+p_2^2-\eps_1)(p_3^2+p_2^2-\eps_2)$, so the vanishing set contains two circles in the $p_2p_3$-plane with radii $\sqrt{\eps_1}$ and $\sqrt{\eps_2}$.
    
    In the second case the Fresnel polynomial becomes $\eps_1(p_1^2-\eps_2)^2$, yielding the points $(\pm\sqrt{\eps_2},0,0)$.
    
    All points described in the third case are contained in $\Sigma_{\eps}$ because
    \begin{equation}
        P_\eps(p_1,p_2,p_3)
        =
        (p_3^2+p_2^2+p_1^2-\eps_2)
        [\eps_2(p_2^2+p_3^2)+\eps_1p_1^2-\eps_1\eps_2]
        .
    \end{equation}
    Therefore $R(\Sigma_{\eps})$ consists of a sphere of radius $\sqrt{\eps_2}$ and a circle in the $p_2p_3$-plane of radius $\sqrt{\eps_1}$, both centered at the origin.
\end{proof}

\begin{lemma}
    \label{lma:eps-1-eigenvalue}
    If $\eps$ has only one eigenvalue, then $R(\Sigma_{\eps})=\Sigma_{\eps}$ is a single sphere whose radius is the square root of the sole eigenvalue.
\end{lemma}

\begin{proof}
    With $\eps=\eps_1I$, we have
    \begin{equation}
        P_\eps(p_1,p_2,p_3)
        =
        \eps_1(p_3^2+p_2^2+p_1^2-\eps_1)^2
        .
    \end{equation}
    The Fresnel surface is a sphere and all points on it are radial. 
    Even if we were to construe $\Sigma_{\eps}$ as a scheme instead of a variety so that all points are singular, the claim would still remain valid.
\end{proof}

\subsection{Uniqueness Results}

\begin{proof}[Proof of Theorem~\ref{thm:eps-uniqueness}]
    The assumption $\Sigma_{\eps_1}=\Sigma_{\eps_2}$ implies that $R(\Sigma_{\eps_1}) = R(\Sigma_{\eps_2})$.
    We will show that this implies $\eps_1=\eps_2$.
    
    For each $\eps_i$, one of the cases of Lemmas~\ref{lma:eps-3-eigenvalues}, \ref{lma:eps-2-eigenvalues}, and~\ref{lma:eps-1-eigenvalue} must hold.
    The geometries of the radial sets $R(\Sigma_{\eps_i})$ are different in each case, so the degeneracies of the spectra of $\eps_i$ are determined by the radial points of the corresponding Fresnel surface.
    In each case the set $R(\Sigma_{\eps_i})$ also determines the eigenvalues and the eigenspaces, so $\eps_i$ is indeed fully determined by $\Sigma_{\eps_i}$, and we have indeed that $\eps_1=\eps_2$.
\end{proof}

\begin{proof}[Proof of Corollary~\ref{thm:eps-uniqueness-subset}]
    By Theorem~\ref{thm:reducibility-characterization} the Fresnel polynomial is irreducible over $\C$ under the assumptions. 
    A nonempty Euclidean open subset of the smooth real locus is Zariski dense in the complex surface $\V(P_{\eps,I})_\C$.
    Therefore, the Zariski closure of the Euclidean open set is the whole Fresnel surface.
    The claim then follows from Theorem~\ref{thm:eps-uniqueness}.
\end{proof}

\section{The Non-Magnetoelectric Case I: Irreducibility Results}
\label{sec:nonmagnetoelectric_irreducibility}

We (re-)prove that, generically, Fresnel polynomials are irreducible over $\C$. 
As we noted before, this result already follows from their connection to Kummer surfaces, but the methods below are of an entirely different character.

Let $A$ be a finitely generated $\mathbb{R}$-algebra and let $R = A[p_0, p_1, p_2, p_3]$ be a polynomial ring in $4$ variables with coefficients in $A$. 
The \emph{homogenized Fresnel polynomial} $\widetilde{P}_{\eps, \mu} \in R$ is obtained by setting
\begin{equation}
    \label{eq:homogenized_fresnel_polynomial}
    \widetilde{P}_{\eps, \mu}(p_0,p_1,p_2,p_3) := p_{0}^{-2}\det(C_{p}-p_{0} M) = p_{0}^4 \,P_{\eps, \mu}\left(\frac{p_1}{p_0},\frac{p_2}{p_0},\frac{p_3}{p_0}\right).
\end{equation}

Specialize to the case $A = \mathbb{R}[\eps_{1}, \ldots, \eps_{6}, \mu_{1}, \ldots, \mu_{6}]$, where the $\eps_{i}, \mu_{i}$ are indeterminates. 
The ring $A$ is a free $\mathbb{R}$-algebra on $12$ generators. 
The homogenized Fresnel polynomial~\eqref{eq:homogenized_fresnel_polynomial} can be viewed as a homogeneous polynomial of degree~$4$ in the graded ring $A[p_0, p_1, p_2, p_3]$, where $A$ itself is a polynomial ring in $12$ variables. 
Its associated affine variety $\mathbb{A}^{12}_\mathbb{R} = \spec A$ is a parameter space for block-diagonal material parameters and hence a family of Fresnel polynomials. 
We are interested in the locus $Y \subset \mathbb{A}^{12}_\mathbb{R}$ where $\det(M) \neq 0$.  
This is an irreducible dense open subset, being the complement of the hypersurface $\mathbb{V}(\det(M)) \subset \mathbb{A}^{12}_\mathbb{R}$, which parametrizes nondegenerate block-diagonal material matrices. 

Let $H_d = \textrm{H}^0\left(\PP^3,\mathscr{O}_{\PP^3}(d)\right) = \R[p_0,p_1,p_2,p_3]_d$ be the $\R$-vector space of homogeneous polynomials of degree $d$ in the polynomial ring $\R[p_0,p_1,p_2,p_3]$. 
For $a + b = 4$, multiplication of polynomials gives a map of projective spaces
\begin{align*}
    m_{a,b} \colon \PP(H_a) \times \PP(H_b) &\longrightarrow \PP(H_4) \\
    \left([f],[g]\right) &\longmapsto [fg]
\end{align*}
Since the source is projective, this map is closed.

\begin{theorem}
    \label{thm:irr_locus}
    Let 
    \begin{align*}
        \ev \colon Y &\longrightarrow \PP(H_4) \\
        (\eps,\mu) &\longmapsto \left[ \widetilde P_{\eps,\mu}\right ]
    \end{align*}
    be the evaluation map, taking a nondegenerate block-diagonal material matrix $\diag(\eps,\mu)$ to its homogenized Fresnel polynomial. 
    Then the locus 
    \begin{align*}
        \irr(Y) :&= \{(\eps,\mu) \in Y : \widetilde P_{\eps,\mu} \text{ is geometrically integral}\} \\
        &= \{(\eps,\mu) \in Y : \V(\widetilde P_{\eps,\mu})_\C \subset \PP^3_\C\text{ is integral}\}
    \end{align*}
    is Zariski open.
\end{theorem}

\begin{proof}
    A homogeneous quartic polynomial is reducible over an algebraically closed field if and only if its factorization is of type $1 + 3$ or $2+2$. 
    Hence, the reducible locus of quartic polynomials in $\PP(H_4)$ is the union of the scheme-theoretic images $\im(m_{1,3})$ and $\im(m_{2,2})$, which is a closed subset; call its open complement $U$. 
    The map evaluation $\ev$ is regular, so $\ev^{-1}(U)$ is open in $Y$.
\end{proof}

We deduce Theorem~\ref{thm:generic-irreducibility} as a corollary.

\begin{proof}[Proof of Theorem~\ref{thm:generic-irreducibility}]
    Since $Y$ is irreducible, every nonempty open subset is dense. 
    We show that $\irr(Y) \neq \emptyset$ in Example~\ref{ex:irreducible_fresnel_polynomial} by exhibiting a homogenized Fresnel polynomial that is irreducible over $\C$. 
    Since $P_{\eps,\mu}$ is a degree $4$ dehomogenization of $\widetilde P_{\eps,\mu}$, one is irreducible over $\C$ if and only if the other one is.
\end{proof}

\begin{example}
    \label{ex:irreducible_fresnel_polynomial}
    Let
    \begin{equation}
        M =
        \begin{pmatrix}
            2&0&1&0&0&0\\
            0&2&0&0&0&0\\
            1&0&1&0&0&0\\
            0&0&0&2&1&0\\
            0&0&0&1&4&0\\
            0&0&0&0&0&2
       \end{pmatrix}.
    \end{equation}
    The Fresnel polynomial associated to this matrix is
    \begin{equation}
        \begin{split}
            P_{\eps,\mu}(p_1,p_2,p_3)
                &=2p_3^4+4p_1p_3^3+8p_2^2p_3^2+2p_1p_2p_3^2+6p_1^2p_3^2-16p_3^2+8p_1p_2^2p_3\\
                &+4p_1^2p_2p_3-8p_2p_3+4p_1^3p_3-16p_1p_3+8p_2^4+4p_1p_2^3+12p_1^2p_2^2\\
                &-46p_2^2+4p_1^3p_2-16p_1p_2+4p_1^4-23p_1^2+28.
        \end{split}
    \end{equation}
    A \textsc{Magma} calculation shows that the homogenization $\widetilde P_{\eps,\mu}$ of this polynomial is irreducible over the finite field $\mathbb{F}_{5^4}$ with $5^4$ elements; applying~\cite[Lemma~19]{tonyjoonas2023} shows this calculation suffices to establish irreducibility over $\C$. 
    The accompanying file {\tt 1-exampleIrreducibleFresnel.m} in~\cite{this-paper} contains \textsc{magma} code verifying the claims in this example.
\end{example}

\begin{remark}
    Theorem~\ref{thm:generic-irreducibility} can also be proved by using the methods of~\cite{tonyjoonas2023}.  
    Indeed, this was our original approach to the theorem, until an audit of a complete draft of the paper by GPT5.6-Sol Pro suggested the above simpler proof.
\end{remark}

\section{The Non-Magnetoelectric Case II: Generic Uniqueness}
\label{sec:nonmagnetoelectric_generic_uniqueness}

In this section, we prove Theorem~\ref{thm:generic-uniqueness}. 
It requires working knowledge of modern algebraic geometry~\cite{Hartshorne,vakil2025} and geometric invariant theory~\cite{Mukai}.

The Fresnel polynomial now depends on $\eps$ and $\mu$. 
Our first task is to describe a scheme (or variety) structure on $X$, parametrizing Fresnel surfaces, that incorporates the gauge freedom~\eqref{eq:gauge}. 
This requires attempting to take a quotient of an algebraic variety by a group action. 
In general, the underlying topological space of such a quotient, in the algebraic category, need not coincide with the set-theoretic quotient of the action, because the action's orbits might have intersecting closures, which must be identified in the algebraic category. 
However, we construct a geometric invariant quotient $\mathcal{X}^\circ$ such that $X$ injects into the set of real-valued points $\mathcal{X}^\circ(\R)$ and is Zariski dense in $\mathcal{X}^\circ$.

\subsection{The Algebraic Quotient}
\label{sec:algebraic_quotient}

Let $E := \Sym_3(\R)\times \Sym_3(\R)$ be the $12$-dimensional $\R$-vector space whose elements are pairs $(\eps,\mu)$ of real symmetric $3\times 3$ matrices, and let 
\[
    V := \mathbb{A}(E) = \spec\left(\Sym_\R(E^\vee)\right) \simeq \mathbb{A}^{12}_\R
\]
be the affine variety associated with $E$; we have $V(\R) = E$. 
Let $C_2 = \langle\tau\rangle$ be a group of order $2$. 
The algebraic group $G := \mathbb G_{m,\R}\rtimes C_2$ acts on $V$ by
\[
    t\cdot(\varepsilon,\mu)=(t\varepsilon,t^{-1}\mu) \text{ for }t\in \mathbb{G}_{m,\R},
    \quad\text{and}\quad
    \tau\cdot(\varepsilon,\mu)=(\mu,\varepsilon).
\]
The semidirect product structure is determined by $\tau t\tau = t^{-1}$. 
Note that $G(\R) = \R^\times \rtimes C_2$, and the intersection of a $G(\R)$-orbit with the physical locus consists of one physical gauge class.

The group~$G$ is reductive, so the invariant ring $\R[V]^G$ is finitely generated, and the affine good quotient
\[
    \mathcal{X} := V \sslash G := \spec \R[V]^G
\]
exists (see, e.g.,~\cite{Mukai}). 
Let $\Delta(\eps,\mu) = \det(\eps)\det(\mu)$; this is a $G$-invariant function on $V$, so we can form the quotient of the open subset $V^\circ := D(\Delta) = \{\Delta \neq 0\} \subset V$ by $G$ to obtain an open subset 
\[
    \mathcal X^\circ := V^\circ \sslash G \subset \mathcal X.
\]
This quotient is equipped with a natural quotient morphism $\pi_G \colon V^\circ \to \mathcal X^\circ$ of schemes.

We claim that $\pi_G \colon V^\circ \to \mathcal X^\circ$ is a geometric quotient, i.e., $\mathcal X^\circ$ is a space whose geometric points are precisely the $G$-orbits of $V^\circ$. 
To see this, it suffices to show that for any algebraically closed field extension $K/\R$, the $G_K$-orbits in $V^\circ_K$ are closed. 
We have
\[
    V \simeq \spec \R[\eps_1,\dots,\eps_6,\mu_1,\dots,\mu_6];
\]
let $P := (e_1,\dots,e_6,m_1,\dots,m_6) \in V^\circ(K)$ be a $K$-valued point. 
There are indices $i$ and $j$ such that $e_i \neq 0$ and $m_j \neq 0$. 
The $\mathbb{G}_{m}(K)$-orbit of $P$ is the vanishing set
\[
    \V(\eps_i\mu_j - e_im_j) \cap \V(e_i\eps_k - e_k\eps_i, m_j\mu_\ell - m_\ell\mu_j ; 1\leq k, \ell \leq 6) \subset V^\circ_K,
\]
which is manifestly closed. 
The full $G$-orbit of $P$ is the union of its $\mathbb{G}_{m}(K)$-orbit and its $\tau$-translate, so it is a finite union of closed sets and hence closed.

Let
\[
    \mathcal P :=\{(\eps,\mu)\in V(\R):\eps\text{ and }\mu\text{ are positive definite}\}.
\]
Then $\mathcal P\subset V^\circ(\R)$ is a nonempty Euclidean open subset, and hence a Zariski dense subset.

\begin{lemma}
    \label{lem:physical-orbits-in-quotient}
    Let $(\eps,\mu),(\eps',\mu')\in\mathcal P$.  
    If $\pi_G(\eps,\mu)=\pi_G(\eps',\mu')$, then there is a $t \in \R_{>0}$ such that
    \[
        (\eps',\mu')=(t\eps,t^{-1}\mu)
        \qquad\text{or}\qquad
        (\eps',\mu')=(t\mu,t^{-1}\eps).
    \]
    In particular, the physical gauge classes $X$ inject into $\mathcal X^\circ(\R)$, and their image is Zariski dense.
\end{lemma}

\begin{proof}
    By the discussion preceding the lemma, equality of points in $\mathcal X^\circ$ means that $(\eps,\mu)$ and $(\eps',\mu')$ belong to the same geometric $G$-orbit.  
    Thus, one of the two displayed relations holds with $t\in\C^\times$.  
    Since $\eps$, $\mu$, $\eps'$, and $\mu'$ all give nonzero real matrices, we must have $t\in\R^\times$. 
    Finally, positive definiteness forces $t>0$.  
    The density assertion follows from the Zariski density of $\mathcal P$ in $V$ and the surjectivity of $\pi_G$.
\end{proof}

Let $W \cong \mathbb{R}^{22}$ be the $\R$-vector space whose vectors record the coefficients of monomials that can occur in a Fresnel polynomial. 
Since the constant term of $P_{\eps,\mu}$ is $\Delta(\eps,\mu)$, any Fresnel polynomial arising from a point on $V^\circ$ is not the zero polynomial. 
This allows us to define a map
\begin{align*}
    \widetilde F \colon V^\circ &\longrightarrow \PP(W) \\
    (\eps,\mu) &\longmapsto [\widetilde P_{\eps,\mu}].
\end{align*}
The composition $q\circ F$ from~\S\ref{sec:main-results} can be understood as the map $\widetilde{F}$ on the $\R$-valued points of $V^\circ$. 
Proposition~\ref{prop:fresnel_gauge_freedom} shows that $\widetilde F$ is $G$-invariant.  
Hence it factors uniquely through the quotient:
\begin{equation}
    \label{eq:quotient-Fresnel-map}
    \begin{tikzcd}
        V^\circ \arrow[r,"\widetilde F"] \arrow[d,"\pi_G"'] & \PP(W) \\
        \mathcal X^\circ \arrow[ur,"\phi"'] & {}
    \end{tikzcd}
\end{equation}
Thus $\widetilde F=\phi\circ\pi_G$.  
Let $\mathcal Y :=\overline{\widetilde F(V^\circ)} \subset\PP(W)$ with its reduced induced structure.  

We obtain a dominant morphism $\phi\colon\mathcal X^\circ\to\mathcal Y$. 
Our goal is to show that this map is a birational morphism of algebraic varieties, and hence generically injective. 
This will imply that a real Fresnel polynomial, up to a global scaling, generically determines the parameters of the electromagnetic material that gave rise to it, up to the gauge~\eqref{eq:gauge}. 
However, working directly with $\mathcal X^\circ$ is computationally challenging, so in the next section, we first find a suitable replacement by `slicing' $V^\circ$ in just the right way as to ``kill'' the $\mathbb{G}_{m,\R}$-action.

\subsection{A $\mathbb{G}_{m,\R}$-slice of $V^\circ$}
\label{sec:slicing}

Let $D(\eps_1) \subset V \simeq \mathbb{A}^{12}_\R$ be the open subset where $\eps_1 \neq 0$, and set $V_1^{\circ} := V^\circ \cap D(\eps_1)$. 
This is a $\mathbb{G}_{m,\R}$-invariant subset of $V^\circ$ containing the positive-definite locus $\mathcal P$. 
Let $S := V^\circ \cap \V(\eps_1 - 1)$ be the ``slice'' of $V_1^\circ$ where $\eps_1 = 1$. 
Every $\mathbb{G}_{m,\R}$-orbit of $V_1^\circ$ contains a unique representative in $S$. 
More precisely, the map
\begin{align*}
    \mathbb{G}_{m,\R} \times S &\longrightarrow V_1^\circ \\
    \left(t,(\eps,\mu)\right) &\longmapsto (t\eps,t^{-1}\mu)
\end{align*}
is an isomorphism; its inverse is $(\eps,\mu) \mapsto \left(\eps_1,(\eps/\eps_1,\eps_1\mu)\right)$.

Unfortunately, the swap $\tau$ does not preserve $S$. 
However, on the dense open subset $S \cap D(\mu_1)$ of $S$, swapping followed by returning to the slice gives a rational involution
\begin{align*}
    j\colon S &\dashrightarrow S \\
    (\eps,\mu) &\longmapsto \left(\frac{\mu}{\mu_1},\mu_1\eps\right)
\end{align*}
whose fixed locus $\{(\eps,\mu) \in S : \mu = \mu_1\eps\}$ is a proper closed subset of $S\cap D(\mu_1)$. 
It follows that a general $G$-orbit of $V^\circ$ meets the slice $S$ in only two distinct points: the $\mathbb{G}_{m,\R}$-orbit of $P \in S$ meets $S$ only at $P$, while the $\mathbb{G}_{m,\R}$-orbit of $\tau(P)$ meets $S$ only at $j(P)$.

The map $j$ pays mathematical dividends: it allows us to make the quotient $\mathcal X^\circ$ in two steps: we can take $\mathbb{G}_{m,\R}$-invariants first, followed by taking $\langle j\rangle$-invariants. 
This is made precise in Theorem~\ref{thm:Fresnel_polynomial_map_birational}, where we show the equality of function fields 
\[
    \R(\mathcal X^\circ)
    =\R(V^\circ)^G
    =\bigl(\R(V^\circ)^{\mathbb{G}_{m,\R}}\bigr)^{\langle\tau\rangle}
    \simeq\R(S)^{\langle j\rangle}
    =\R(\mathcal Y).
\]

\subsection{Projective Compactification and Blowing-up}
\label{sec:compactification_and_blow-up}

In order to study the fibres of $\widetilde{F}$ algebraically, we first extend this map to a projective setting. Write the coordinates on $\PP^6\times \PP^6$ as
\[
    ([\eps_1,\ldots,\eps_6,r],[\mu_1,\ldots,\mu_6,s]).
\]
The affine chart where $r = s = 1$ is naturally isomorphic to $V$. 
Bi-homogenising the coefficients of the Fresnel polynomial, we obtain a rational map
\[
    g\colon \PP^6\times \PP^6\dashrightarrow \PP(W).
\]
More precisely, if a coefficient of the affine Fresnel polynomial has bi-degree $(a,b)$ in $(\eps,\mu)$, then it is multiplied by the appropriate powers of $r$ and $s$ so that it becomes bihomogeneous of degree $(3,3)$.
For instance, the term $\eps_1\mu_1$ is replaced by $\eps_1 r^2\mu_1s^2$. 
The restriction of $g$ to the affine chart $r = s = 1$ agrees with $\widetilde F$.

The projective closure of the slice $S$ is the subvariety
\[
    \Pi=\{\eps_1=r\}\times\PP^6 \subset\PP^6\times\PP^6.
\]
We use it to perform a blow-up calculation. 
Let $B\subset \Pi$ be the scheme-theoretic base locus of $g|_\Pi$, and let 
\[
    \beta\colon X' := \text{Bl}_B(\Pi) \to \Pi
\]
be the blow-up morphism. 
Concretely, $X'$ is the projective closure of the graph of $g|_\Pi$, and $\beta$ resolves $g|_\Pi$ in the sense that there exists a morphism $h\colon X' \to \mathcal Y$ fitting into the commuting triangle
\begin{equation}
    \begin{tikzcd}
        X' \arrow[swap]{d}{\beta} \arrow{dr}{h} & {} \\
        \Pi \arrow[dashed,swap]{r}{g|_\Pi} & \mathcal{Y}
    \end{tikzcd}.
\end{equation}
The variety $X'$ is integral and projective, and the morphism $h$ is proper and surjective: for this last property, note that $h(X')$ is closed, it contains $\widetilde F(S)$ and $\widetilde F(S)$ is dense in $\mathcal Y$.

The key to proving birationality of the map $\phi\colon \mathcal X^\circ \to \mathcal Y$ is showing there is a single point $y_0 \in \mathcal Y(\R)$ such that $h^{-1}(y_0)$ consists of two distinct reduced $\R$-points. 
We record the existence of such a point in the following example.

\begin{example}
    \label{ex:special-fibre}
    Fix the ordered monomial basis
    \[
        \begin{aligned}
            \mathcal B=\bigl(&
            p_1^4,p_1^3p_2,p_1^3p_3,p_1^2p_2^2,p_1^2p_2p_3,p_1^2p_3^2,p_0^2p_1^2,p_1p_2^3,p_1p_2^2p_3,p_1p_2p_3^2, \\
            &p_0^2p_1p_2, p_1p_3^3,p_0^2p_1p_3,p_2^4,p_2^3p_3,p_2^2p_3^2,p_0^2p_2^2,p_2p_3^3,p_0^2p_2p_3,p_3^4,p_0^2p_3^2,p_0^4\bigr)
        \end{aligned}
    \]
    of $W$, and let
    \[
        \begin{aligned}
            y_0=\bigl[  &200:140:120:700:100:1000:-3181:280:240:420:-1000:\\
                        &360:-870:600:200:1700:-5656:300:-696:1200:-8109:12528
                \bigr]\in\PP(W).
        \end{aligned}
    \]
    Equivalently, $y_0=[Q_0]$, where
    \[
        \begin{aligned}
        Q_0={}  & 200p_1^4+140p_1^3p_2+120p_1^3p_3+700p_1^2p_2^2+100p_1^2p_2p_3+1000p_1^2p_3^2+280p_1p_2^3+240p_1p_2^2p_3\\
                &+420p_1p_2p_3^2+360p_1p_3^3+600p_2^4+200p_2^3p_3+1700p_2^2p_3^2+300p_2p_3^3+1200p_3^4-3181p_0^2p_1^2\\
                &-1000p_0^2p_1p_2-870p_0^2p_1p_3-5656p_0^2p_2^2-696p_0^2p_2p_3-8109p_0^2p_3^2+12528p_0^4.
        \end{aligned}
    \]
    The scheme-theoretic fibre $h^{-1}(y_0)$ is reduced of length two.  
    Its support lies in the affine slice $S$ and consists of the two rational points
    \[
        x_1=(\eps_1,\mu_1),
        \qquad\text{and}\qquad
        x_2=(\eps_2,\mu_2),
    \]
    where
    \[
        \eps_1=
        \begin{pmatrix}
            1&7/20&3/10\\
            7/20&3/2&1/4\\
            3/10&1/4&2
        \end{pmatrix},
        \qquad
        \mu_1=
        \begin{pmatrix}
            2&0&0\\
            0&4&0\\
            0&0&6
        \end{pmatrix},
    \]
    and
    \[
        \eps_2=
        \begin{pmatrix}
            1&0&0\\
            0&2&0\\
            0&0&3
        \end{pmatrix},
        \qquad
        \mu_2=
        \begin{pmatrix}
            2&7/10&3/5\\
            7/10&3&1/2\\
            3/5&1/2&4
        \end{pmatrix}.
    \]
    In other words, we also have $\widetilde F(x_1) = \widetilde F(x_2) = y_0$. 
    Moreover, $\eps_2=\mu_1/2$, and $\mu_2=2\eps_1$, so $x_2=j(x_1)$, and all four matrices are positive definite. 
    The accompanying file {\tt 2-biGradedBlowup.magma} in~\cite{this-paper} contains \textsc{magma} code verifying the claims in this example.
\end{example}

\begin{theorem}[Birationality of the Fresnel polynomial map]
    \label{thm:Fresnel_polynomial_map_birational}
    The morphism $\widetilde F|_S \to \mathcal Y$ is generically finite of degree $2$, and at the level of function fields, we have
    \[
        \R(\mathcal Y) = \R(S)^{\langle j\rangle} = \R(\mathcal X^\circ).
    \]
    In particular, the morphism $\phi \colon \mathcal X^\circ \to \mathcal Y$ is birational.
\end{theorem}

\begin{proof}
    The fibre $h^{-1}(y_0)$ in Example~\ref{ex:special-fibre} is zero-dimensional. 
    Therefore, upper semicontinuity of fibre dimension for surjective proper morphisms gives an open subset of $\mathcal Y$ over which $h$ is quasi-finite. 
    The restriction of $h$ to the preimage of this open set is therefore proper and quasi-finite, hence finite. 
    For a finite morphism, the fibre-degree function $y\mapsto \dim_{\kappa(y)} \operatorname{H}^0\big(X'_y,\mathscr{O}_{X'_y}\big)$ is also an upper semicontinuous function. 
    Since the fibre $h^{-1}(y_0)$ has degree $2$, we deduce that at the level of function fields, we have the inequality of degrees
    \[
        [\R(S):\R(\mathcal Y)] = [\R(X'):\R(\mathcal Y)] \leq 2,
    \]
    where the first equality uses the fact that $S$ and $X'$ are birational. 
    See~~\cite[\S5.5]{tonyjoonas2023} for more details of how an argument like this works.

    On the other hand, the nontrivial rational involution $j$ of $S$ satisfies $\widetilde F|_S\circ j=\widetilde F|_S$.  
    It therefore induces a nontrivial automorphism of $\R(S)$ fixing $\R(\mathcal Y)$.  
    Consequently, $[\R(S):\R(\mathcal Y)] \geq 2$, and hence equality holds.  
    Since $\R(\mathcal Y)\subseteq\R(S)^{\langle j\rangle}$ and both subfields have index two in $\R(S)$, we obtain $\R(\mathcal Y)=\R(S)^{\langle j\rangle}$.

    Finally, we have $\R(V^\circ)^{\mathbb{G}_{m,\R}}\simeq\R(S)$. 
    Under this identification, the residual action of $\tau$ is the involution $j$.  
    Because $\pi_G$ is a geometric quotient, pullback identifies the function field of $\mathcal X^\circ$ with the field of $G$-invariant rational functions on $V^\circ$.  
    Therefore,
    \[
        \R(\mathcal X^\circ)
        =\R(V^\circ)^G
        =\bigl(\R(V^\circ)^{\mathbb{G}_{m,\R}}\bigr)^{\langle\tau\rangle}
        \simeq\R(S)^{\langle j\rangle}
        =\R(\mathcal Y).
        \eqno\qed
    \]
    \hideqed
\end{proof}

\begin{remark}
    Mathematically speaking, it would have been more natural to work directly with the map $\phi\colon \mathcal{X}^\circ \to \mathcal Y$ in the proof above, and try to show its general fibre is zero-dimensional of degree $1$. 
    However, to write down $\mathcal{X}^\circ$, one needs a set of coordinate $G$-invariant functions.  
    Let $s_{ij} = \eps_i\mu_j + \eps_j\mu_i$, where $1 \leq i \leq j \leq 6$, and let $S = (s_{ij})$ be a $6\times 6$ symmetric matrix collecting these $G$-invariant functions. 
    One can show that
    \[
        \R[V]^G \simeq \frac{\R[s_{ij} : 1 \leq i \leq j \leq 6]}{I_3(S)},
    \]
    where $I_3(S)$ is the ideal generated by the $3\times 3$ minors of $S$. 
    With this presentation, $\mathcal{X}^\circ$ naturally lives in $\mathbb{A}^{21}$. 
    From here, we would have to take an open subset, homogenize, blow up, and look at a fibre to run the above argument.  
    This blow-up seems to be beyond current computational capabilities.
\end{remark}

\subsection{From Polynomials to Surfaces}
\label{sec:polys_to_surfaces}

\begin{lemma}
    \label{lem:real-surface-determines-quartic}
    Let $F,G\in\R[p_0,p_1,p_2,p_3]$ be nonzero homogeneous polynomials of the same degree.  
    Suppose that $F$ is irreducible over $\C$ and that $\V(F)$ has a smooth real point.  
    If $G$ vanishes on a nonempty Euclidean open subset of the smooth real locus of $\V(F)$, then $G=cF$ for some $c\in\R^\times$.
\end{lemma}

\begin{proof}
    A Euclidean neighborhood of a smooth real point in $\V(F)(\R)$ is Zariski dense in the complex hypersurface $\V(F)_{\C}$. 
    Thus $G$ vanishes identically on $\V(F)_{\C}$, so the irreducible polynomial $F$ divides $G$ in $\C[p_0,p_1,p_2,p_3]$.  
    Since $F$ and $G$ have the same degree, they differ by a nonzero complex scalar.  
    Because both have real coefficients, that scalar is real.
\end{proof}

By Theorem~\ref{thm:generic-irreducibility}, there is a nonempty $G$-invariant Zariski-open subset $V^{\mathrm{gi}}\subset V^\circ$ on which homogenized Fresnel polynomials $\widetilde P_{\eps,\mu}$ are geometrically integral.  
Consider the polynomial condition
\[
    \mathfrak d(\eps,\mu) :=\operatorname{Disc}_{t}\bigl(P_{\eps,\mu}(t,0,0)\bigr)\neq0;
\]
it is $G$-invariant because the Fresnel polynomial is $G$-invariant.
This open condition is nonempty: for the material parameter in Example~\ref{ex:irreducible_fresnel_polynomial},
\[
    P_{\eps,\mu}(t,0,0) = 4t^4-23t^2+28 = 4(t^2-4)(t^2-7/4),
\]
which has four distinct roots.  
Set
\begin{equation}
\label{eq:surface-good-locus}
    V^{\mathrm{surf}} := V^{\mathrm{gi}}\cap D(\mathfrak d) \subset V^\circ.
\end{equation}
This is a nonempty $G$-invariant Zariski-open subset.  
If $(\eps,\mu)\in V^{\mathrm{surf}}(\R)\cap\mathcal P$, then the Fresnel surface has a smooth real point.  
Indeed, let $M=\diag(\eps,\mu)$, let $u=(1,0,0)$, and put
\[
    A := M^{-1/2}C_uM^{-1/2}.
\]
The matrix $A$ is real symmetric and has inertia tuple $(2,2,2)$.  
Since
\[
    P_{\eps,\mu}(tu) = \det(M)\det(tA-I),
\]
its four roots are the reciprocals of the four nonzero eigenvalues of $A$, and are therefore real.  
Since $\mathfrak d \neq 0$, these roots are distinct.  
At each such root the derivative in the $p_1$-direction is nonzero, so the corresponding point of the full Fresnel surface is smooth.

Since $V^{\mathrm{surf}}$ is $G$-invariant, its image
\[
    \mathcal X^{\mathrm{surf}} := \pi_G(V^{\mathrm{surf}})\subset\mathcal X^\circ
\]
is a nonempty Zariski-open subset.  
Lemma~\ref{lem:real-surface-determines-quartic} therefore gives the following consequence.

\begin{corollary}
    \label{cor:surface-equality-implies-coefficient-equality}
    Suppose $(\eps,\mu)\in\mathcal P$ has quotient class lying in $\mathcal X^{\mathrm{surf}}(\R)$, and let $(\eps',\mu')\in\mathcal P$.  
    If $\Sigma_{\eps,\mu}=\Sigma_{\eps',\mu'}$, then $\widetilde F(\eps,\mu)=\widetilde F(\eps',\mu')$ in $\PP(W)$.
    \qed
\end{corollary}

\newpage

\subsection{Generic Reconstruction from the Real Fresnel surface}
\label{sec:generic_reconstruction_proof}

\begin{proof}[Proof of Theorem~\ref{thm:generic-uniqueness}]
    By Theorem~\ref{thm:Fresnel_polynomial_map_birational}, there exists a nonempty Zariski-open subset $\mathcal Y_0\subset\mathcal Y$ such that $\phi^{-1}(\mathcal Y_0) \xrightarrow{\ \sim\ }\mathcal Y_0$ is an isomorphism. 
    Set
    \begin{equation}
        \label{eq:final-good-open-quotient}
        \Omega := \mathcal X^{\mathrm{surf}} \cap \phi^{-1}(\mathcal Y_0) \subset\mathcal X^\circ.
    \end{equation}
    Both factors are nonempty Zariski-open subsets of the irreducible variety $\mathcal X^\circ$, so $\Omega$ is nonempty.  
    Since the image of the positive-definite locus is Zariski dense by Lemma~\ref{lem:physical-orbits-in-quotient}, $\Omega(\R)$ contains physical gauge classes.

    Suppose that $(\eps,\mu)\in\mathcal P$ has quotient class in $\Omega(\R)$, and suppose that $(\eps',\mu')\in\mathcal P$ satisfies $\Sigma_{\eps,\mu}=\Sigma_{\eps',\mu'}$. 
    By Corollary~\ref{cor:surface-equality-implies-coefficient-equality} we know that $\widetilde F(\eps,\mu) = \widetilde F(\eps',\mu')$.
    Equivalently,
    \[
            \phi\bigl(\pi_G(\eps,\mu)\bigr) = \phi\bigl(\pi_G(\eps',\mu')\bigr).
    \]
    The common image lies in $\mathcal Y_0$, and since $\phi^{-1}(\mathcal Y_0)\to\mathcal Y_0$ is an isomorphism, we deduce that $\pi_G(\eps,\mu)=\pi_G(\eps',\mu')$.
    By Lemma~\ref{lem:physical-orbits-in-quotient}, there is a $t>0$ such that either $(\eps',\mu')=(t\eps,t^{-1}\mu)$ or $(\eps',\mu')=(t\mu,t^{-1}\eps)$. 
    This is the gauge freedom~\eqref{eq:gauge}.

    Finally, the inverse image $\pi_G^{-1}(\Omega)\subset V^\circ\subset V$ is a nonempty $G$-invariant Zariski-open subset.  
    Its complement is a proper real algebraic subset of $V$.  This gives the exceptional algebraic locus in the formulation of Theorem~\ref{thm:generic-uniqueness} stated in the introduction.
\end{proof}

\printbibliography

@unpublished{tonyjoonas2023,
    author = {Maarten V. de Hoop and Joonas Ilmavirta and Matti Lassas and Anthony V{\'a}rilly-Alvarado},
    title = {{Reconstruction of generic anisotropic stiffness tensors from partial data around one polarization}},
    month = jul,
    year = {2023},
	url={http://users.jyu.fi/~jojapeil/pub/slowness-ag-1.pdf},
	arxiv = {2307.03312}
}

@unpublished{this-paper,
    author = {Antonio Cocan and Maarten V. de Hoop and Joonas Ilmavirta and Matti Lassas and Anthony V{\'a}rilly-Alvarado},
    title = {{Generic Recovery of Permittivity and Permeability in Anisotropic Maxwell Systems}},
    month = aug,
    year = {2026},
    note = {Preprint, arXiv (this article). See supplementary files on arXiv for magma code.}
}

@book {vakil2025,
    AUTHOR = {Vakil, Ravi},
     TITLE = {The rising sea---foundations of algebraic geometry},
 PUBLISHER = {Princeton University Press, Princeton, NJ},
      YEAR = {[2025] \copyright 2025},
     PAGES = {xxiii+662},
      ISBN = {978-0-691-26866-8; 978-0-691-26867-5; 978-0-691-26868-2},
   MRCLASS = {14-01},
  MRNUMBER = {4942570},
}

@incollection{calderon1980inverse,
  author    = {Calder{\'o}n, Alberto P.},
  title     = {On an inverse boundary value problem},
  booktitle = {Seminar on Numerical Analysis and its Applications to Continuum Physics},
  pages     = {65--73},
  publisher = {Sociedade Brasileira de Matem{\'a}tica},
  address   = {Rio de Janeiro},
  year      = {1980}
}

@article{sylvesteruhlmann1987global,
  author  = {Sylvester, John and Uhlmann, Gunther},
  title   = {A global uniqueness theorem for an inverse boundary value problem},
  journal = {Annals of Mathematics},
  volume  = {125},
  number  = {1},
  pages   = {153--169},
  year    = {1987},
  doi     = {10.2307/1971291}
}

@article{somersalo1992linearized,
  author  = {Somersalo, Erkki and Isaacson, David and Cheney, Margaret},
  title   = {A linearized inverse boundary value problem for {Maxwell}'s equations},
  journal = {Journal of Computational and Applied Mathematics},
  volume  = {42},
  pages   = {123--136},
  year    = {1992}
}

@article{sunuhlmann1992maxwell,
  author  = {Sun, Ziqi and Uhlmann, Gunther},
  title   = {An inverse boundary value problem for {Maxwell}'s equations},
  journal = {Archive for Rational Mechanics and Analysis},
  volume  = {119},
  pages   = {71--93},
  year    = {1992},
  doi     = {10.1007/BF00376011}
}

@article{ola1993electrodynamics,
  author  = {Ola, Petri and P{\"a}iv{\"a}rinta, Lassi and Somersalo, Erkki},
  title   = {An inverse boundary value problem in electrodynamics},
  journal = {Duke Mathematical Journal},
  volume  = {70},
  number  = {3},
  pages   = {617--653},
  year    = {1993},
  doi     = {10.1215/S0012-7094-93-07014-7}
}

@article{ola1996sommerfeld,
  author  = {Ola, Petri and Somersalo, Erkki},
  title   = {Electromagnetic inverse problems and generalized {Sommerfeld} potentials},
  journal = {SIAM Journal on Applied Mathematics},
  volume  = {56},
  number  = {4},
  pages   = {1129--1145},
  year    = {1996},
  doi     = {10.1137/S0036139995283948}
}

@article{carozhou2014global,
  author  = {Caro, Pedro and Zhou, Ting},
  title   = {On global uniqueness for an {IBVP} for the time-harmonic {Maxwell} equations},
  journal = {Analysis \& PDE},
  volume  = {7},
  number  = {2},
  pages   = {375--405},
  year    = {2014},
  doi     = {10.2140/apde.2014.7.375}
}

@article{dossantosferreira2009limiting,
  author  = {Dos Santos Ferreira, David and Kenig, Carlos E. and Salo, Mikko and Uhlmann, Gunther},
  title   = {Limiting {Carleman} weights and anisotropic inverse problems},
  journal = {Inventiones Mathematicae},
  volume  = {178},
  number  = {1},
  pages   = {119--171},
  year    = {2009},
  doi     = {10.1007/s00222-009-0196-4}
}

@article{kenig2011anisotropicmaxwell,
  author  = {Kenig, Carlos E. and Salo, Mikko and Uhlmann, Gunther},
  title   = {Inverse problems for the anisotropic {Maxwell} equations},
  journal = {Duke Mathematical Journal},
  volume  = {157},
  number  = {2},
  pages   = {369--419},
  year    = {2011},
  doi     = {10.1215/00127094-1272903}
}

@article{greenleaf2007fullwave,
  author  = {Greenleaf, Allan and Kurylev, Yaroslav and Lassas, Matti and Uhlmann, Gunther},
  title   = {Full-wave invisibility of active devices at all frequencies},
  journal = {Communications in Mathematical Physics},
  volume  = {275},
  number  = {3},
  pages   = {749--789},
  year    = {2007},
  doi     = {10.1007/s00220-007-0311-7}
}

@article{greenleaf2009invisibility,
  author  = {Greenleaf, Allan and Kurylev, Yaroslav and Lassas, Matti and Uhlmann, Gunther},
  title   = {Invisibility and inverse problems},
  journal = {Bulletin of the American Mathematical Society},
  volume  = {46},
  number  = {1},
  pages   = {55--97},
  year    = {2009},
  doi     = {10.1090/S0273-0979-08-01232-9}
}

@book{hudson1905kummer,
author = {Hudson, Ronald W. H. T.},
title = {Kummer's Quartic Surface},
publisher = {Cambridge University Press},
year = {1905}
}

@article{bateman1910kummer,
author = {Bateman, Harry},
title = {Kummer's Quartic Surface as a Wave Surface},
journal = {Proceedings of the London Mathematical Society},
volume = {8},
number = {1},
pages = {375--382},
year = {1910}
}

@book {hehlobukhov2003,
    AUTHOR = {Hehl, Friedrich W. and Obukhov, Yuri N.},
     TITLE = {Foundations of classical electrodynamics},
    SERIES = {Progress in Mathematical Physics},
    VOLUME = {33},
      NOTE = {Charge, flux, and metric},
 PUBLISHER = {Birkh\"auser Boston, Inc. Boston MA},
      YEAR = {2003},
     PAGES = {xvi+410},
      ISBN = {0-8176-4222-6},
   MRCLASS = {78-01 (53C80 78A25)},
  MRNUMBER = {1999743},
MRREVIEWER = {A.\ P.\ Stone},
       DOI = {10.1007/978-1-4612-0051-2},
       URL = {https://doi-org.ezproxy.rice.edu/10.1007/978-1-4612-0051-2},
}

@article{lindellfavarobergamin2011,
author = {Favaro, Alberto and Bergamin, Lucio and Lindell, Ismo V.},
title = {The Optical Geometry of Wave Propagation in Linear Media},
journal = {International Journal of Geometric Methods in Modern Physics},
volume = {8},
number = {7},
pages = {1495--1503},
year = {2011}
}

@article{bergaminfavaro2014,
author = {Favaro, Alberto and Bergamin, Lucio},
title = {The Non-birefringent Limit of All Linear, Skewonless Media and Kummer Surfaces},
journal = {Annalen der Physik},
volume = {523},
number = {5},
pages = {383--401},
year = {2011}
}

@article{favaro2011,
author = {Favaro, Alberto},
title = {Recent Advances in the Description of Electromagnetic Media Using Fresnel Surfaces},
journal = {Acta Physica Polonica A},
volume = {124},
number = {2},
pages = {240--244},
year = {2013}
}

@article{FavaroHehl2016,
  title = {Light propagation in local and linear media: Fresnel-Kummer wave surfaces with 16 singular points},
  author = {Favaro, Alberto and Hehl, Friedrich W.},
  journal = {Phys. Rev. A},
  volume = {93},
  issue = {1},
  pages = {013844},
  numpages = {6},
  year = {2016},
  month = {Jan},
  publisher = {American Physical Society},
  doi = {10.1103/PhysRevA.93.013844},
  url = {https://link.aps.org/doi/10.1103/PhysRevA.93.013844}
}

@article{favaro2016,
author = {Favaro, Alberto and Hehl, Friedrich W.},
title = {Light Propagation in Local and Linear Media: Fresnel-Kummer Wave Surfaces with 16 Singular Points},
journal = {Physical Review A},
volume = {93},
number = {1},
pages = {013844},
year = {2016}
}

@article{Baekler2014,
  author  = {Peter Baekler and Alberto Favaro and Yakov Itin and Friedrich W. Hehl},
  title   = {The Kummer tensor density in electrodynamics and in gravity},
  journal = {Annals of Physics},
  volume  = {349},
  pages   = {297--324},
  year    = {2014},
  month   = oct,
  doi     = {10.1016/j.aop.2014.06.007},
  issn    = {0003-4916},
  eprint  = {1403.3467},
  archivePrefix = {arXiv},
  primaryClass  = {gr-qc}
}

@article {Dolgachev,
    AUTHOR = {Dolgachev, I.},
     TITLE = {Kummer surfaces: 200 years of study},
   JOURNAL = {Notices Amer. Math. Soc.},
  FJOURNAL = {Notices of the American Mathematical Society},
    VOLUME = {67},
      YEAR = {2020},
    NUMBER = {10},
     PAGES = {1527--1533},
      ISSN = {0002-9920,1088-9477},
   MRCLASS = {14J28 (14-02 14-03)},
  MRNUMBER = {4201885},
       DOI = {10.1090/noti},
       URL = {https://doi-org.ezproxy.rice.edu/10.1090/noti},
}

@book {Hartshorne,
    AUTHOR = {Hartshorne, Robin},
     TITLE = {Algebraic geometry},
    SERIES = {Graduate Texts in Mathematics},
    VOLUME = {No. 52},
 PUBLISHER = {Springer-Verlag, New York-Heidelberg},
      YEAR = {1977},
     PAGES = {xvi+496},
      ISBN = {0-387-90244-9},
   MRCLASS = {14-01},
  MRNUMBER = {463157},
MRREVIEWER = {Robert\ Speiser},
}

@book {Mukai,
    AUTHOR = {Mukai, Shigeru},
     TITLE = {An introduction to invariants and moduli},
    SERIES = {Cambridge Studies in Advanced Mathematics},
    VOLUME = {81},
   EDITION = {Japanese},
 PUBLISHER = {Cambridge University Press, Cambridge},
      YEAR = {2003},
     PAGES = {xx+503},
      ISBN = {0-521-80906-1},
   MRCLASS = {14-02 (14D20 14L24 14L30)},
  MRNUMBER = {2004218},
MRREVIEWER = {Arvid\ Siqveland},
}

@incollection {magma,
    AUTHOR = {Bosma, Wieb and Cannon, John and Playoust, Catherine},
     TITLE = {The {M}agma algebra system. {I}. {T}he user language},
      NOTE = {Computational algebra and number theory (London, 1993)},
   JOURNAL = {J. Symbolic Comput.},
  FJOURNAL = {Journal of Symbolic Computation},
    VOLUME = {24},
      YEAR = {1997},
    NUMBER = {3-4},
     PAGES = {235--265},
      ISSN = {0747-7171,1095-855X},
   MRCLASS = {68Q40},
  MRNUMBER = {1484478},
       DOI = {10.1006/jsco.1996.0125},
       URL = {https://doi-org.ezproxy.rice.edu/10.1006/jsco.1996.0125},
}

@article{Dencker,
  title={On the propagation of polarization sets for systems of real principal type},
  author={Dencker, Nils},
  journal={Journal of Functional Analysis},
  volume={46},
  number={3},
  pages={351--372},
  year={1982},
  publisher={Elsevier}
}

\end{document}